\documentclass[11pt,oneside,reqno]{amsart}
\usepackage{amssymb}
\usepackage{setspace}
 \usepackage[foot]{amsaddr}

\newcommand{\N}{\mathbb{N}}
\newcommand{\Z}{\mathbb{Z}}

\newcommand{\R}{\mathbb{R}}

\newcommand{\sk}{\mathcal{S}^k}
\newcommand{\car}{\mathbb{S}}
\newcommand{\gas}{\mathbb{G}}

\newcommand{\eps}{\varepsilon}
\newcommand{\pip}{\varphi}

\newcommand{\ma}{\mu_\alpha}
\newcommand{\magas}{\mu'_\alpha}

\newcommand{\wa}{\omega_\alpha}
\newcommand{\rad}{\text{rad}}
\newcommand{\Lip}{\text{Lip}\,}

\def\vint_#1{\mathchoice
          {\mathop{\vrule width 6pt height 3 pt depth -2.5pt
                  \kern -8pt \intop}\nolimits_{#1}}%
          {\mathop{\vrule width 5pt height 3 pt depth -2.6pt
                  \kern -6pt \intop}\nolimits_{#1}}%
          {\mathop{\vrule width 5pt height 3 pt depth -2.6pt
                  \kern -6pt \intop}\nolimits_{#1}}%
          {\mathop{\vrule width 5pt height 3 pt depth -2.6pt
                  \kern -6pt \intop}\nolimits_{#1}}}

\newcommand{\norm}[1]{\left\lVert#1\right\rVert}
\DeclareMathOperator{\diam}{diam}
\DeclareMathOperator{\dist}{dist}

\newtheorem{thm}{Theorem}[section]
\newtheorem{lemma}[thm]{Lemma}

\theoremstyle{definition}
\newtheorem{defn}[thm]{Definition}
\newtheorem*{notn}{Notation}

\begin{document}

\title[Traces and extensions for weighted Sobolev spaces]
{Traces and extensions of certain weighted Sobolev spaces on $\mathbb{R}^n$ and Besov functions on Ahlfors regular compact 
subsets of $\mathbb{R}^n$}
\author{Jeff Lindquist and Nageswari Shanmugalingam}
\thanks{The second author was partially supported by grant DMS~\#1800161 from the National Science Foundation 
(U.S.A.). Both authors also
acknowledge a great debt to the custodial staff at the University of Cincinnati who maintained the facilities during these
difficult times; this debt can never be adequately repaid.}

\dedicatory{Dedicated with gratitude to Professor Pekka Koskela on his 59th birthday}

\begin{abstract}
The focus of this paper is on Ahlfors $Q$-regular compact sets $E\subset\mathbb{R}^n$ such that, for each
$Q-2<\alpha\le 0$, the weighted measure $\mu_{\alpha}$ given by integrating the density $\omega(x)=\text{dist}(x, E)^\alpha$
yields a Muckenhoupt $\mathcal{A}_p$-weight in a ball $B$ containing $E$. For such sets $E$ we show 
the existence of a bounded linear trace operator acting from $W^{1,p}(B,\mu_\alpha)$ to 
$B^\theta_{p,p}(E, \mathcal{H}^Q\vert_E)$ when $0<\theta<1-\tfrac{\alpha+n-Q}{p}$, and the
existence of a bounded linear extension operator from $B^\theta_{p,p}(E, \mathcal{H}^Q\vert_E)$
to $W^{1,p}(B, \mu_\alpha)$ when $1-\tfrac{\alpha+n-Q}{p}\le \theta<1$.
We illustrate these results with $E$ as the Sierpi\'nski carpet, the Sierpi\'nski 
gasket, and the von Koch snowflake.  
\end{abstract}

\maketitle




\vskip 1cm

\section{Introduction}

The Sobolev classes $W^{1,p}(\Omega)$, with $\Omega$ a Euclidean domain, are 
associated with (potentially degenerate)
elliptic partial differential equations related to a strongly local Dirichlet form.
In considering Dirichlet boundary value problems on a Euclidean domain in $\R^n$ related to elliptic differential operators, solutions are 
known to exist if the prescribed boundary datum, namely a function that is given on the boundary of the domain, 
arises as the trace of a Sobolev function that is defined on the domain. The works of Jonsson and Wallin~\cite{JW, JW2} identify
certain Besov spaces of functions on a compact $d$-set as traces on that set of Sobolev functions on $\R^n$. Here, a 
set is a $d$-set if it is Ahlfors $d$-regular, namely, $\mathcal{H}^d(B(x,r))\simeq r^d$ whenever $x$ is a point in that set and
$r>0$ is no larger than the diameter of that set. If $\Omega$ is a bounded Lipschitz domain in $\mathbb{R}^n$, then its boundary
is an $(n-1)$-set. It was shown in~\cite[Page~228]{JW} that the trace of $W^{1,p}(\R^n)$ in a $d$-set $E\subset \mathbb{R}^n$ is
the Besov space $B^{1-(n-d)/p}_{p,p}(E)$. Thus for each $p>1$ a specific value of $\theta=1-\tfrac{n-d}{p}$ was chosen
amongst all possible values of $0<\theta<1$ for which the Besov space $B^\theta_{p,p}(E)$ is identified with the Sobolev
space $W^{1,p}(\mathbb{R}^n)$.

On the other hand,~\cite[Theorem~1.1]{B2S} tells us that each compact doubling metric measure space $Z$ is the boundary
of a uniform space (a locally compact non-complete metric space that is a uniform domain in its completion) $X_\eps$ and
that for each choice of $p>1$ and $0<\theta<1$ there is a choice of a doubling measure $\mu_\beta$,
$\beta=\beta(\theta)$, on the uniform space such that $B^\theta_{p,p}(Z)$ is the trace space of the 
Newton-Sobolev class $N^{1,p}(X_\eps,\mu_\beta)$. The measures $\mu_\beta$ are obtained as weighted measures,
that is, $d\mu_\beta(x)=\omega_\beta(x) \, d\mu(x)$ for some underlying measure $\mu$ on $X_\eps$. With this perspective
in mind, it is then natural to ask whether, given $p>1$ and $0<\theta<1$, there is a weighted measure $\mu_\alpha$
on $\R^n$, with $\alpha$ depending perhaps on $\theta$ and $p$, such that $B^\theta_{p,p}(E,\nu)$ is the trace space
of the weighted Sobolev class $W^{1,p}(\R^n, \mu_\alpha)$. Here $\nu$ is the Hausdorff measure on $E$
of dimension $\text{dim}_H(E)$.
There is evidence for this to be plausible, see for example~\cite{Maly} where $E$ was  considered to be the boundary
of a uniform domain. 
The paper~\cite[Theorem~1.5]{Barton} identified each $B^\theta_{p,p}(\partial \Omega)$, $1<p<\infty$, $0<\theta<1$, with certain 
weighted Sobolev classes of functions on the Lipschitz domain $\Omega\subset\R^n$, while
the papers~\cite{DFM1, DFM2}
consider the case of $p=2$ with $\R^n\setminus E$ a domain satisfying a one-sided NTA domain condition. 
In this note we do not assume that $\R^n\setminus E$ is a domain (the example of $E$ being a Sierpi\'nski
carpet does not have its complement in $\R^2$ be a domain) nor do we restrict ourselves to the case of $p=2$.
Indeed, when $n\ge 3$, in
considering $\R^n$, the choice of $p=n$ is related to the quasiconformal geometry of $\R^n\setminus E$.
In many cases, including the three examples considered here as well as the setting considered in
\cite{Maly, DFM1, DFM2}, the weight $\omega_\beta(x)\simeq\dist(x,E)^\beta$.

In this paper we expand on this question and study the following setting.
With $0<Q<n$, let $E\subset \R^n$ be an Ahlfors $Q$-regular compact set
with $\diam(E)\le 1$, and $B$ be a ball in $\R^n$ such that $E\subset\tfrac12 B$. We also assume that 
for each $\alpha\le 0$ there is a measure $\mu_\alpha$ on $B$ such that whenever $\alpha+n-Q>0$
and $x\in E$, and $0<r<2$ such that $r>\dist(x,E)/9$, the comparison $\mu_\alpha(B(x,r))\simeq r^{n+\alpha}$
holds. We also assume that the ball $B$, equipped with the Euclidean metric $d$ and the measure $\mu_\alpha$,
is doubling. Furthermore, we assume that for each $p>1$ there exists such $\alpha$ so that $\ma$
supports a $p$-Poincar\'e inequality. These assumptions are not as restrictive as one might think. From 
Theorem~1.1 of~\cite{DILTV}, we know that the measure $\mu_\alpha$ given by $d\mu_\alpha=\dist(x,E)^\alpha dm(x)$
with $m$ the $n$-dimensional Lebesgue measure on $\R^n$ is a Muckenhoupt $\mathcal{A}_p$-weight and hence
is doubling and supports a $p$-Poincar\'e inequality for all $1<p<\infty$. Moreover, from the proof of
Theorem~3.4 of~\cite{DILTV} we also automatically have $\mu_\alpha(B(x,r))\lesssim r^{n+\alpha}$; thus
the only requirement we add to this discussion is that $r^{n+\alpha}\lesssim\mu_\alpha(B(x,r))$.

We prove the following two theorems in this paper.
In what follows, $\nu=\mathcal{H}^Q\vert_E$ and for each real number $\alpha$ there is a Borel regular measure
$\ma$ on $\R^n$ that is absolutely continuous with respect to the Lebesgue measure $m$ and satisfying
$\mu_\alpha(B(x,r))\simeq r^{n+\alpha}$ when $x\in E$ and $0<r\le 10$.

\begin{thm}\label{thm:main-trace}
Let $p>1$ and  $0<\theta<1$ such that $p\theta<1$. 
Let $\alpha\le 0$ such that $\alpha+n-Q>0$ and
$\theta<1-\tfrac{\alpha+n-Q}{p}$ such that $\ma$ is doubling and supports a $p$-Poincar\'e inequality.
Then there is a bounded linear trace operator 
\[
T:N^{1,p}(B,\mu_\alpha)\to B^\theta_{p,p}(E,\nu)
\] 
such that $Tu=u\vert_E$ when $u$ is a Lipschitz function on $B$.
\end{thm}

\begin{thm}\label{thm:main-extend}
Let $p>1$ and $0<\theta<1$, and let $\alpha\le 0$ such that
$\alpha+n-Q>0$ and $\theta\ge 1-\tfrac{\alpha+n-Q}{p}$. Then there is a bounded
linear extension operator
\[
S:B^\theta_{p,p}(E,\nu)\to N^{1,p}(B,\mu_\alpha)
\]
such that if $u$ is a Lipschitz function on $E$, then $Su$ is a Lipschitz function on $B$ with
$u=Su\vert_E$.
\end{thm}

As one might note above, there is a lack of sharpness in the range of allowable $\theta$ in Theorem~\ref{thm:main-trace},
which then prevents us from identifying $B^\theta_{p,p}(E,\nu)$ as the trace space of $N^{1,p}(B,\mu_\alpha)$
when $\theta=1-\tfrac{\alpha+n-Q}{p}$. This hurdle is overcome if $E$ is the boundary of a uniform domain,
as shown by Mal\'y in~\cite{Maly}, see also~\cite{DFM1}. 
We do not know whether we can take $\theta=1-\tfrac{\alpha+n-Q}{p}$ in
Theorem~\ref{thm:main-trace}. 

In this note we also show that the fractals such as the
standard Sierpi\'nski carpet, the standard Sierpin\'ski gasket, and the von Koch snowflake
curve in $\R^2$
satisfy the conditions imposed on the set $E$ above, with $n=2$. For such sets, the measure $\mu_\alpha$ is 
a weighted Lebesgue measure given in~\eqref{eq:def-mu-alpha} below.  
The trace and extension theorems for the example of the von Koch snowflake curve follow from the
work of~\cite{Maly} once a family of measures $\ma$, $0\ge \alpha>-\beta_0$ for a suitable $\beta_0$ are
constructed and shown to satisfy the hypotheses of~\cite[Theorem~1.1]{Maly}. We do this construction in
Section~\ref{VonKoch}.
We show also that for the Sierpi\'nski
carpet and the Sierpi\'nski gasket, this weight
$\omega(x)=\dist(x,E)^\alpha$ is a Muckenhoupt $\mathcal{A}_q$-weight for each $q>1$ when $\alpha+2>Q$, where
$Q$ is the Hausdorff dimension of the fractal set. 
As mentioned above, the Muckenhoupt $\mathcal{A}_p$ criterion follows from~\cite[Theorem~1.1]{DILTV}, but we 
give a constructive proof for these fractals and in addition provide a proof of the co-dimensionality condition.
Observe that if a weight is an $\mathcal{A}_q$-weight, then
the associated weighted measure is doubling and supports a $q$-Poincar\'e inequality. However, not all
weights that give a doubling measure supporting a $q$-Poincar\'e inequality are $\mathcal{A}_q$-weights, see
the discussion in~\cite{HKM}.

Another interesting nonlocal space is the so-called Haj\l asz space, see for example~\cite{HM, HKSTbook}.
In~\cite[Theorem~9]{HM} it is shown that if $\Omega\subset \R^n$ satisfies an $A(c)$-condition (a porosity condition at the boundary),
then the trace of Sobolev spaces $W^{1,p}(\Omega)$
are Haj\l asz spaces of functions on $E$ when $E$ is equipped with an appropriately snow-flaked metric and 
a doubling Borel measure. 
We refer the interested reader to \cite{JW,JW2, Besov, Maz, KSW, Barton, HM} 
for more on Sobolev spaces and Besov spaces of functions
on Euclidean domains and sets, to \cite{GKS, GKZ, Maly} for more on Besov spaces of functions on
subsets of certain metric measure spaces, but this is not an exhaustive list of papers on these topics in the current literature.

\section{Background}

In this section we describe the notions used throughout this note. 
 With $0<Q<n$, let $E\subset \R^n$ 
be an Ahlfors $Q$-regular compact set
with $\diam(E)\le 1$ and let $B$ be a fixed ball in $\R^n$ such that $E\subset\tfrac12 B$.
We set $\nu=\mathcal{H}^Q\vert_E$. We also consider the measure $\ma$, obtained as a weighted measure with
respect to the Lebesgue measure on $\R^n$ as in~\cite{HKM}, namely $d\ma(x)=\omega(x)\, dm(x)$ with
$m$ the canonical Lebesuge measure on $\R^n$.

There are two basic function spaces under consideration
here: the weighted Sobolev space $W^{1,p}(B,\ma)$ and the Besov space $B^\theta_{p,p}(E,\nu)$ with $1<p<\infty$. 
Recall from~\cite{HKM} that a function $f\in L^p(B,\ma)$ is in the weighted Sobolev space $W^{1,p}(B,\ma)$
if $f$ has a weak derivative $\nabla f:B\to\R^n$ such that $|\nabla f|\in L^p(B,\ma)$. It was shown in~\cite{HKM}
that when $\ma$ is doubling and satisfies a $p$-Poincar\'e inequality near 
on $B$, then $W^{1,p}(B,\ma)$ is a Banach space. Here, $W^{1,p}(B,\ma)$ is equipped with the norm
\[
\Vert f\Vert_{N^{1,p}(B,\ma)}:=\Vert f\Vert_{L^p(B,\ma)}+\Vert |\nabla f|\Vert_{L^p(B,\ma)}.
\]
Recall that $\ma$ is doubling on $B$ if there is a constant $C\ge 1$ such that whenever $x\in B$ and $0<r\le 10$,
\[
\ma(B(x,2r))\le C\, \ma(B(x,r)),
\]
and we say that $\ma$ supports a $p$-Poincar\'e inequality on $B$ if there is a constant $C>1$ such that 
whenever $f$ is in $W^{1,p}(B,\ma)$ and
$B_0$ is a ball with $B_0\cap B\ne \emptyset$ and $\rad(B_0)\le 10$, then 
\[
\vint_{B_0}|f-f_{B_0}|\, d\ma\le C \rad(B_0)\, \left(\vint_{B_0}|\nabla f|^p\, d\ma\right)^{1/p}.
\]
Here 
\[
f_{B_0}:=\vint_{B_0}f\, d\ma=\frac{1}{\ma(B_0)}\int_{B_0} f\, d\ma.
\]
A function $u\in L^p(E,\nu)$ is in the Besov space $B^\theta_{p,p}(E,\nu)$ for a fixed $0<\theta<1$ if
\[
\Vert u\Vert_{B^\theta_{p,p}(E,\nu)}^p
:=\int_E\int_E \frac{|u(y)-u(x)|^p}{d(x,y)^{\theta p}\nu(B(x,d(x,y)))}\, d\nu(y)\, d\nu(x)
\]
is finite.

The next two notions are related to the specific examples considered in this paper.
Information about Muckenhoupt weights can be found for example in~\cite{MW1, MW2, HKM},
while information about uniform domains can be found for example in~\cite{MS, HrK, BS}; these example
references barely scratch the surface of the current literature on these topics and therefore should not be 
considered to be even an almost exhaustive list.

\begin{defn}
For $p > 1$, a weight $\omega \colon \R^n \to [0,\infty)$ is said to be a 
\emph{Muckenhoupt $\mathcal{A}_p$-weight near $B$} if 
$\omega$ is a  locally integrable function such that
\[
\sup_{B_0} \biggl( \frac{1}{m(B_0)} \int_{B_0} \omega dm \biggr) 
 \biggl(\frac{1}{m(B_0)} \int_{B_0} \omega^{-\tfrac{q}{p}} dm \biggr)^{\tfrac{p}{q}} < \infty
\]
where $\tfrac{1}{p} + \tfrac{1}{q} = 1$, and $B_0$ ranges over balls in 
$\R^n$ intersecting $B$ with $\rad(B_0)\le 10$.  In this case we write $\omega \in \mathcal{A}_p$.
\end{defn}

\begin{defn}
A domain $\Omega\subset\R^n$ is said to be a uniform domain if $\Omega\ne\R^n$ and there is a constant 
$A\ge 1$ such that for each distinct pair of points $x,y\in\Omega$ there is a curve $\gamma$ in $\Omega$
with end points $x,y$ with length $\ell(\gamma)\le A\, d(x,y)$ and 
$\min\{\ell(\gamma_{z,x}),\ell(\gamma_{z,y})\}\le A\, \dist(z,\partial\Omega)$ whenever $z$ is a point on $\gamma$.
Here $\gamma_{z,y}$ denotes each subcurve of $\gamma$ with end points $z,y$.
\end{defn}

\section{The Carpet and Gasket examples}

In this section we first consider the weighted measure 
corresponding to the weight $\omega_\alpha(x)=\dist(x,\car)^\alpha$,
with $\car$ the standard Sierpi\'nski carpet with outer perimeter $\partial [0,1]^2$.  We will show that
when $\alpha>Q-2=\tfrac{\log(8)}{\log(3)}-2$, the weighted measure is doubling. We also show that
if in addition $p>\tfrac{\alpha+2-Q}{2-Q}$, then the weight is a Muckenhoupt $\mathcal{A}_p$-weight.

For $\alpha \in \R$, we define a Borel measure $\ma$ with density $\omega_\alpha(x) = \dist(x, \car)^\alpha$ 
outside of $\car$.  That is, for Borel sets $A\subset \R^2$ we have
\begin{equation}\label{eq:def-mu-alpha}
\ma(A) = \int_A \dist(x,\car)^\alpha dx.
\end{equation}
Note that $\car$ has Lebesgue measure zero.  We now investigate for which values of $\alpha$ the measure $\ma$ is doubling.

\begin{notn}
The carpet $\car$ can be written as $[0,1]^2 \setminus \bigcup_i H_i$ where the collection $\bigcup_i H_i$ consists of 
pairwise disjoint open squares $H_i\subset[0,1]^2$.  We call the open squares $H_i$ holes.  Each hole has 
sidelength $3^{-k}$ for some $k \in \N$.  If the sidelength of a hole $H$ is $3^{-k}$, then we say $H$ 
belongs to generation $k$.  For each $k\in\N$ there are $8^{k-1}$ holes in generation $k$. 
\end{notn}

\begin{lemma}\label{hole lemma}
Let $H$ be a hole in generation $k$.  If $\alpha > -1$, then $\ma(H) = c_\alpha 3^{-k(\alpha+2)}$ where 
$c_\alpha = \frac{8}{(\alpha+1)(\alpha+2)} 2^{-(\alpha + 2)}$ .  Otherwise, $\ma(H) = \infty$.
\end{lemma}

\begin{proof}
By symmetry, we have 
\[
\ma(H) = 8 \int_0^{2^{-1}3^{-k}} \int_0^x y^\alpha dy dx.  
\]
If $\alpha \leq -1$, then $\ma(H)$ is infinite.  
Otherwise, we have $\alpha + 1 > 0$ and so
\[
\ma(H) = \frac{8}{\alpha+1} \int_0^{2^{-1}3^{-k}} x^{\alpha + 1} dx 
= \frac{8\cdot\, 2^{-(\alpha+2)}}{(\alpha+1)(\alpha+2)} (3^{-k})^{\alpha + 2} = c_\alpha 3^{-k(\alpha+2)}.
\]
\end{proof}

For our analysis we use squares instead of Euclidean balls. 
For $x \in \R^2$ and $s>0$, we set 
\[
S(x, s) = \{y \in \R^2 : \norm{x - y}_\infty < \tfrac{s}{2} \},
\]
the open square in $\R^2$ centered at $x$ with sidelength $s$.  For $\tau > 0$, we set $\tau\, S(x,s) = S(x,\tau s)$
and estimate $\ma(S(x,s))$.  
To do this, we introduce families of squares that have easy to compute $\ma$ mass.  For $k \in \N$, let 
$\sk$ be the set of (open) squares of the form $S(x, 3^{-k})$ with $x$ of the form 
$((m+ \tfrac{1}{2})3^{-k}, (n+ \tfrac{1}{2})3^{-k})$ with $m, n \in \Z$.

\begin{lemma}\label{sk square lemma}
Let $S = S(x, s) \in \sk$ and $\alpha > \tfrac{\log(8)}{\log(3)} - 2$.  If 
$\overline{9S} \cap \car \neq \emptyset$, then $\ma(S) \simeq s^{\alpha + 2}$.  
Otherwise, $\ma(S) \simeq s^2 d(x, \car)^\alpha$.  In particular, if $c > 9$ is the 
smallest integer such that $\overline{cS} \cap \car \neq \emptyset$, then 
$\ma(S) \simeq c^\alpha s^{\alpha + 2}$
\end{lemma}

Observe that when $\alpha> \tfrac{\log(8)}{\log(3)} - 2$, we automatically have $\alpha>-1$.

\begin{proof}
For $\alpha = 0$ the claim is clear, so we assume that $\alpha \neq 0$ for the remainder of the proof.  
Note that $s = 3^{-k}$ as $S \in \sk$.  First suppose that $\overline{9S} \cap \car \neq \emptyset$.  
We examine three cases: (i) $3^k (\overline{S}\cap \car)$ is isometric to the carpet, (ii) $S$ is a hole as above, or (iii) 
neither of these cases.  

\textbf{Case (i):}  Using Lemma \ref{hole lemma}, we compute $\ma(S)$ exactly.  For each 
$j \in \N_0$, we see that $S$ contains $8^j$ holes of generation $k+j+1$.  By assumption, 
$3^{\alpha + 2} > 8$.
As $S$ is a scaled copy of the carpet, it follows from Lemma \ref{hole lemma} that 
\begin{equation*}
\begin{split}
\ma(S) = \sum_{j=0}^{\infty} 8^j c_\alpha (3^{\alpha + 2})^{-k-j-1} &
= c_\alpha (3^{\alpha + 2})^{-k-1} \sum_{j=0}^\infty 8^j 3^{-j(\alpha+2)} \\
&= c_\alpha 3^{-k(\alpha + 2)}\frac{1}{3^{\alpha + 2} - 8} 
= \biggl(\frac{c_\alpha}{3^{\alpha+2} - 8}\biggr) s^{\alpha + 2}.
\end{split}
\end{equation*}

\textbf{Case (ii):} In this case $S$ is a hole in generation $k$, so by Lemma \ref{hole lemma} we have 
that $\ma(S) = c_\alpha s^{\alpha + 2}$.

\textbf{Case (iii):} First assume that $\alpha < 0$.  From our choice of $\sk$, in this case we must have 
that $S \cap \car = \emptyset$. It is clear that if $H$ is a hole of generation $k$ and $\iota \colon S \to H$ 
is an isometry given by translation, then for all $y \in S$ we have $d(y, \car) \geq d(\iota(y), \car)$.  As 
$\alpha < 0$, it follows that $d(y, \car)^\alpha \leq d(\iota(y), \car)^\alpha$ and so 
$\ma(S) \leq \ma(H) \simeq s^{\alpha + 2}$.  On the other hand, 
$\overline{9S} \cap \car \neq \emptyset$, so $d(y, \car) \leq 11 \cdot 3^{-k} = 11s$ for 
all $y \in S$.  Hence, as $\alpha < 0$ we have $d(y,\car)^\alpha \geq (11s)^\alpha$.  It follows that
\[
\ma(S) \geq \int_S (11s)^\alpha = 11^\alpha s^{\alpha + 2}.
\]
If $\alpha > 0$ instead, then the lower and upper bounds above are reversed but the conclusion is the same.  

Now suppose that $\overline{9S} \cap \car = \emptyset$.  It follows that $d(x, \car) \geq 3s$.  Then, if $y \in S$, we have 
\[
\tfrac{1}{3} d(x, \car) \leq d(x, \car) - s \leq d(y, \car) \leq d(x, \car) + s \leq 2 d(x, \car).
\]
The result follows immediately.

For the last part of the lemma,  if $c > 9$ is the smallest integer such that $\overline{cS} \cap \car \neq \emptyset$, then $d(x, \car) \simeq cs$ and so $\ma(S) \simeq c^\alpha s^{\alpha + 2}$.
\end{proof}

We now use Lemma \ref{sk square lemma} to prove the same result for general squares.

\begin{lemma}\label{gen square lemma}
Let $S = S(x, s)$ with $s \le 9$.  Let $\alpha > \tfrac{\log(8)}{\log(3)} - 2$.  If $\overline{9S} \cap \car \neq \emptyset$, then $\ma(S) \simeq s^{\alpha + 2}$.  Otherwise, $\ma(S) \simeq s^2 d(x, \car)^\alpha$.  In particular, if $c > 9$ is the smallest integer such that $\overline{cS} \cap \car \neq \emptyset$, then $\ma(S) \simeq c^\alpha s^{\alpha + 2}$
\end{lemma}

\begin{proof}
The proof that if $\overline{9S} \cap \car = \emptyset$, then $\ma(S) \simeq s^2 d(x, \car)^\alpha$ is the same as in Lemma \ref{sk square lemma}.  For the first part of the statement of the lemma, suppose that $\overline{9S} \cap \car \neq \emptyset$.  Let $k \in \N$ be the smallest integer with $3^{-k} < s$.  As $s\le 9$ we have $\tfrac{s}{3^{-k}}\le 3$.  
It follows that there is a subset $\{S_i\}_{i \in I} \subseteq S_k$ with $S \subseteq \cup_{i \in I} S_i$ and 
$|I| \le 25$.  Write $S_i = S(x_i, s_k)$ with $s_k = 3^{-k}$.  For each $S_i$ we have 
\[
d(x_i, \car) \leq s_k + s + d(x, \car)
\]
and, as $s_k \simeq s$ and $d(x,\car) \simeq s$, there is a constant $a > 0$ independent of $i$ and $S$ 
such that $\overline{aS_i} \cap \car \neq \emptyset$ for each $i \in I$.  Hence, we may apply 
Lemma~\ref{sk square lemma} to each $S_i$ (with different squares $S_i$ potentially falling into 
different cases of Lemma~\ref{sk square lemma}) and conclude that $\ma(S_i) \simeq s^{\alpha + 2}$.  
Hence, as $|I| \le 25$ we have $\ma(S) \lesssim s^{\alpha + 2}$.  

For the lower bound, we again choose the smallest integer $k$ with $3^{-k} < s$ and then note that 
there is a square $S' \in \mathcal{S}^{k+2}$ with $S' \subseteq S$.  An argument similar to the above 
shows us that $\ma(S') \simeq s^{\alpha + 2}$, and so $\ma(S) \gtrsim s^{\alpha + 2}$.  
\end{proof}

We now show that $\ma$ is doubling for small squares.

\begin{lemma}[Doubling]
Let $S = S(x, s)$ be a square such that $s \le 9$.  Then, $\ma(3S) \lesssim \ma(S)$.  
\end{lemma}

\begin{proof}
First suppose that $\overline{27S} \cap \car = \emptyset$.  Then, by Lemma \ref{gen square lemma} 
we know that $\ma(3S) \simeq (3s)^2 d(x, \car)^\alpha$ and $\ma(S) \simeq s^2 d(x, \car)^\alpha$.  

Now, suppose that $\overline{27S} \cap \car \neq \emptyset$.  Then, by Lemma \ref{gen square lemma}, 
we know that $\ma(3S) \simeq (3s)^{\alpha + 2}$.  We also see from Lemma \ref{gen square lemma} 
that $\ma(S) \simeq s^{\alpha + 2}$ (either $\overline{3S} \cap \car \neq \emptyset$, or we have $c \leq 27$ 
in the last part of the statement of Lemma \ref{gen square lemma}).  
\end{proof}

Note that $q = \frac{p}{p-1}$.  We investigate conditions that guarantee that the weight $\wa$ given by
$\wa(x) = \dist(x, \car)^\alpha$ is in the Muckenhoupt class $\mathcal{A}_p$, see also~\cite[Theorem~1.1]{DILTV}.   
It is clear that in 
the definition of $\mathcal{A}_p$ we may replace the use of balls $B$ with that of squares $S$.

\begin{lemma}[$\mathcal{A}_p$ weights]\label{lem-ap-weights}
The function $\wa(x) = \dist(x, \car)^\alpha$ is a Muckenhoupt $\mathcal{A}_p$-weight near $B$ when 
$\tfrac{\log(8)}{\log(3)} - 2 < \alpha <  (p-1)(2 - \tfrac{\log(8)}{\log(3)})$.
\end{lemma}

\begin{proof}
Let $\tfrac{\log(8)}{\log(3)} - 2 < \alpha <  (p-1)(2 - \tfrac{\log(8)}{\log(3)})$.  Let $S = S(x,s)$ be a square with 
$s < 9$.  First, assume that $\overline{9S} \cap \car \neq \emptyset$.  Then, by Lemma \ref{gen square lemma}, 
\[
\frac{1}{m(B)} \int_B \wa dm \lesssim \frac{1}{s^2} s^{\alpha + 2} = s^\alpha.
\]
We see $\tfrac{q}{p} =\tfrac{1}{p-1}$.  As $-\tfrac{\alpha}{p-1} > \tfrac{\log(8)}{\log(3)} - 2$, by Lemma \ref{gen square lemma} we have
\begin{equation*}
\begin{split}
\biggl(\frac{1}{m(B)} \int_B \wa^{-\tfrac{q}{p}} dm \biggr)^{\tfrac{p}{q}} &= \biggl(\frac{1}{m(B)} \int_B \dist(x,\car)^{-\tfrac{\alpha}{p-1}} dx \biggr)^{p-1} \\
&\lesssim \biggl(\frac{1}{s^2} s^{-\tfrac{\alpha}{p-1} + 2}\biggr)^{p-1} = s^{-\alpha}.
\end{split}
\end{equation*}
It follows that the $\mathcal{A}_p$ bound holds for these squares for $\alpha$ in the above range.

If, instead, we have $\overline{9S} \cap \car = \emptyset$, then with $\alpha$ in the same range we have 
\[
\frac{1}{m(B)} \int_B \wa dm \lesssim \frac{1}{s^2} s^{2}\dist(x,\car)^\alpha =\dist(x,\car)^\alpha
\]
and
\begin{equation*}
\begin{split}
\biggl(\frac{1}{m(B)} \int_B \wa^{-\tfrac{q}{p}} dm \biggr)^{\tfrac{p}{q}} &= \biggl(\frac{1}{m(B)} \int_B \dist(x,\car)^{-\tfrac{\alpha}{p-1}} dx \biggr)^{p-1} \\
&\lesssim \biggl(\frac{1}{s^2} s^{2} \dist(x, \car)^{-\tfrac{\alpha}{p-1}}\biggr)^{p-1} = \dist(x,\car)^{-\alpha}.
\end{split}
\end{equation*}
It again follows that for $\tfrac{\log(8)}{\log(3)} - 2 < \alpha <  (p-1)(2 - \tfrac{\log(8)}{\log(3)})$ the $\mathcal{A}_p$ bound holds.
\end{proof}

Now, let $\gas$ denote the Sierpi\'nski gasket for which the points $(0,0), (1,0)$, and $(\tfrac{1}{2}, \tfrac{\sqrt{3}}{2})$ are the vertices of its boundary triangle.  Consider the measures
\[
\magas(A) = \int_A \dist(x,\gas)^\alpha dx.
\]
As for the carpet, for the correct values of $\alpha$ the measures $\magas$ are doubling and the functions $\dist(x, \gas)^\alpha$ are $\mathcal{A}_p$ weights.  The argument for this is similar, and we summarize the differences below.

First, squares are no longer the natural objects for integration.  Instead, it is easier to work with equilateral triangles.  The grids $\sk$ are replaced by grids of equilateral triangles with side lengths $2^{-k}$.  If $T = T(x, s)$ is such a triangle (centered at $x$ with side length $s$), then for $\alpha > \tfrac{\log(3)}{\log(2)} - 2$ we can estimate $\magas(T)$ as in Lemma \ref{sk square lemma}.  For grid triangles far from the gasket relative to their side lengths, the estimate $\magas(T) \simeq s^2 \dist(x, \gas)^\alpha$ still holds.  For grid triangles near the gasket relative to their side length but which are neither holes nor scaled versions of the gasket, the estimate $\magas(T) \simeq s^{2+\alpha}$ still holds by comparison with $\magas(H)$ for hole-triangles $H$.  For a single hole triangle $T$ with side length $s = 2^{-k}$, we have
\[
\magas(T) = 6 \int_0^{s/2} \int_0^{x / \sqrt{3}} y^\alpha dy dx \simeq s^{\alpha + 2}.
\]
For triangles $T = T(x, s)$ where $T \cap \gas$ is a scaled copy of the gasket, we see that if $s = 2^{-k}$ then for each $j \in \N_0$, the triangle $T$ contains $3^j$ hole triangles of side length $2^{-k-j-1}$.  Hence, in this case
\[
\magas(T) \simeq \sum_{j=0}^\infty 3^j 2^{(-k-j-1)(\alpha + 2)} \simeq  s^{\alpha + 2}
\]
where the series is finite as $2^{\alpha+2} > 3$.  

As in Lemma \ref{gen square lemma}, one again estimates the $\magas$ measure of arbitrary triangles from the grid triangles.  Once this is done, the doubling property and $\mathcal{A}_p$ weight condition are as before.  For the function $\dist(x, \gas)^\alpha$ to be an $\mathcal{A}_p$ weight, we require $\tfrac{\log(3)}{\log(2)} - 2 < \alpha <  (p-1)(2 - \tfrac{\log(3)}{\log(2)})$.

\section{The von Koch snowflake curve}\label{VonKoch}

In this section we consider the von Koch snowflake curve $K$ 
(with an equilateral triangle with side-lengths $1$ as a ``zeroth" iteration)
that is the boundary of the von Koch snowflake domain. We will show that $\Omega$ will satisfy the hypotheses
of our paper with $K$ playing the role of $E$ and the ball $B$ replaced by $\Omega$. However, in this case
we obtain a sharper result by combining the results of~\cite{BS} with~\cite[Theorem~1.1]{Maly} to obtain  the full range
$0<\theta\le 1-\tfrac{\alpha+n-Q}{p}$ in Theorem~\ref{thm:main-trace}.

From~\cite[Theorem~1]{Ahl}  we know that the snowflake curve is a quasicircle. Therefore, 
from~\cite[Theorem~2.15]{MS} or~\cite[Theorem~1.2]{Hr} it follows that the von Koch snowflake domain is a
uniform domain; this sets us up to use the results of~\cite{BS}.

\begin{defn}[Definition 2.6 in \cite{BS}]
Let $\Omega\subset\R^n$ be a domain and 
$\beta > 0$.  We say that $\Omega$ satisfies a local $\beta$-shell condition if there is a constant 
$C > 0$ such that for all $x \in \overline{\Omega}$ and $0 < \rho \leq r \leq \diam(\Omega)$ we have
\begin{equation}\label{beta shell cond}
m(\{y \in B(x,r) \cap \Omega : \delta_\Omega(y) \leq \rho\}) \leq C \biggl( \frac{\rho}{r}\biggr)^\beta m(B(x,r) \cap \Omega)
\end{equation}
where $\delta_\Omega(y) = \dist(y, \partial\Omega)$. 
Recall here that $m$ is the $n$-dimensional Lebesgue measure on $\R^n$.
\end{defn}

\begin{defn}
We say that $\Omega$ satisfies a strong local $\beta$-shell condition if it satisfies a local $\beta$-shell condition
and in addition,
\begin{equation}\label{strong beta shell cond}
m(\{y \in B(x,r) \cap \Omega : \delta_\Omega(y) \leq \rho\}) \simeq \biggl( \frac{\rho}{r}\biggr)^\beta m(B(x,r) \cap \Omega)
\end{equation}
whenever $x\in \partial\Omega$ and $0 < \rho \leq r \leq \diam(\Omega)$.
\end{defn}

By \cite[Lemma 2.7]{BS}, if $\Omega\subset\R^n$ is a bounded domain
that satisfies the above local $\beta$-shell condition for 
some $\beta>0$, then for all 
$\alpha > -\beta$ the measure $d\ma(y)=\delta_\Omega(y)^\alpha dm(y)$ is doubling on $\Omega$.  Combining this 
with~\cite[Theorem~4.4]{BS} and noting that the Newton-Sobolev space discussed there is the same as the
standard Sobolev space $W^{1,p}(\R^2)$ in our setting (see for example~\cite{HKSTbook}) tells us that
when $\Omega$ is the von Koch snowflake domain, the metric measure space
$(\Omega, d, \ma)$ is doubling and supports a $1$-Poincar\'e inequality. Here $d$ is the Euclidean metric. Moreover, 
with $\nu=\mathcal{H}^Q\vert_K$ where $Q$ is the Hausdorff dimension of $K$, we would also
have that $\nu(B(x,r))\simeq r^Q$ whenever $x\in K$ and $0<r\le 10$. 

\begin{lemma}\label{lem:vonKoch-codim}
Suppose that $\Omega\subset\R^n$ a bounded domain that 
satisfies a strong local $\beta$-shell condition for some $\beta>0$ and that
$K=\partial\Omega$ is Ahlfors $Q$-regular for some $n>Q>0$. For $-\beta<\alpha\le 0$ we set
$\ma$ to be the measure on $\Omega$ given by $d\ma(y)=\delta_\Omega(y)^\alpha\, dm(y)$. 
Moreover, assume that for each $x\in K$ and $0<r\le 10$ we have $m(B(x,r)\cap\Omega)\simeq r^n$.
Then  whenever $0<r\le 10$ and $x\in K$, we have 
\[
\ma(B(x,r)\cap\Omega)\simeq r^{n+\alpha-Q}\, \mathcal{H}^Q(B(x,r)\cap K).
\]
\end{lemma}

\begin{proof}
Fix $x\in K$ and $0<r\le 10$, and without loss of generality assume that $\alpha<0$. Then by the Cavalieri principle,
\begin{align*}
\ma(B(x,r)\cap\Omega)&=\int_{B(x,r)\cap\Omega}\frac{1}{\delta_\Omega(y)^{|\alpha|}}\, dm(y)\\
  &=\int_0^\infty m(\{y\in B(x,r)\cap\Omega\, :\, \delta_\Omega(y)^{-|\alpha|}>t\})\, dt\\
  &\simeq \int_0^M Cr^n\, dt+\int_M^\infty m(\{y\in B(x,r)\cap\Omega\, :\, \delta_\Omega(y)<t^{-1/|\alpha|}\})\, dt
\end{align*}
where $M=r^{-|\alpha|}>0$. 
Note that as $\beta<\alpha<0$, we have  $\beta/|\alpha|>1$. 
Using $M=r^{-|\alpha|}$ and the local shell property, we obtain
\begin{align*}
\ma(B(x,r)\cap\Omega)&\simeq r^{n-|\alpha|}+\int_M^\infty \left(\frac{1}{t^{1/|\alpha|}r}\right)^\beta r^n\, dt\\
 &\simeq r^{n+\alpha}+r^{n-\beta}M^{1-\beta/|\alpha|}
 \simeq r^{n+\alpha}.
\end{align*}
Since $\mathcal{H}^Q(B(x,r)\cap K)\simeq r^Q$, the conclusion follows.
\end{proof}

Hence if the von Koch snowflake domain satisfies a strong local $\beta$-shell condition, then $(\Omega, \ma)$
is doubling and supports a $1$-Poincar\'e inequality when $\alpha>-\beta$, and in addition from Lemma~\ref{lem:vonKoch-codim}
we know that $\nu=\mathcal{H}^Q\vert_K$ is $2+\alpha-Q$-codimension regular with respect to $\ma$, and
so by~\cite[Theorem~1.1]{Maly} the conclusions of Theorem~\ref{thm:main-trace} and Theorem~\ref{thm:main-extend}
hold for the von Koch domain and its boundary.

In light of the above discussion, 
it only remains to verify the strong local $\beta$-shell condition for $\Omega$ for the choice of
$0<\beta=\beta_0:= 2 - \tfrac{\log(4)}{\log(3)}=2-Q$.  

For each non-negative integer $n$, let $K_n$ denote the $n$-th 
iteration of the von Koch snowflake, so $K_n$ consists of $3 \cdot 4^n$ line segments of 
length $3^{-n}$.  Let $x \in \overline{\Omega}$, $0<r<\tfrac{1}{2}$, and choose a non-negative integer $k$ 
such that $3^{-k-1} \leq 2r < 3^{-k}$.  For $0<\rho<r$ choose a non-negative integer $j$ such that 
$3^{-k-j-1} \leq \rho < 3^{-k-j}$, and so $\rho \simeq 3^{-j} r$.  

There is a constant $M$, independent of $x$ and $r$, such that the number of the line 
segments in $K_{k+1}$ intersecting $B(x, 2r)$ is at most $M$, for the set 
of endpoints of the segments of $K_{k+1}$ are $3^{-k-1}$-separated and $r \simeq 3^{-k-1}$.  

For $m \in \N$, let $A_m$ be the bounded component of $\R^2 \setminus K_m$.  
Clearly $A_m\subset A_{m+1}$. Moreover,
\[
m(\{y\in\Omega\, :\, \delta_\Omega(y)\le \rho\}\cap B(x,r))
\simeq m(\{y\in A_{k+j+1}\, :\, \delta_\Omega(y)\le \rho\}\cap B(x,r)).
\]
Also, each line segment $L$  of length $3^{-k-1}$
that makes up the construction of $K_{k+1}$ is modified 
$j$ times to obtain the set $K_{k+j+1}$ by replacing $L$ with $4^j$ number of line segments,
each of length $3^{-j-k-1}$. If $\ell$ is one of these line segments, then
\[
m(\{y\in A_{k+j+1}\, :\, \dist(y,\ell)\le \delta_\Omega(y)\le \rho\})\simeq\rho^2,
\]
and therefore
\[
 m(\{y\in A_{k+j+1}\, :\, \delta_\Omega(y)\le \rho\}\cap B(x,r))\lesssim M\times 4^j\times \rho^2,
\]
with $\lesssim$ actually being $\simeq$ if $x\in K$. From the fact that $\rho\simeq 3^{-j}r$, it follows that
\[
m(\{y\in\Omega\, :\, \delta_\Omega(y)\le \rho\}\cap B(x,r))\lesssim \left(\frac{4}{9}\right)^j r^2
 \simeq  \left(\frac{4}{9}\right)^j m(B(x,r)\cap\Omega),
\]
again with $\lesssim$ actually being $\simeq$ if $x\in K$.
Set $\beta_0=2-\tfrac{\log(4)}{\log(3)}=2-Q$ and observe that $\rho/r\simeq 3^{-j}$. 
Then we have that
\[
\left(\frac{4}{9}\right)^j=\left(3^{-j}\right)^{\beta_0}\simeq \left(\frac{\rho}{r}\right)^{\beta_0}
\]
as desired, proving that the snowflake domain satisfies the strong $\beta_0$-shell condition.

\section{Trace of weighted Sobolev functions are in Besov spaces, or how the surrounding leaves a mark on subsets}

The goal of this section is to study trace of $N^{1,p}(B,\mu_\alpha)$ on the set $E$ and relate it to the
Besov classes $B^\theta_{p,p}(E,\nu)$ and prove Theorem~\ref{thm:main-trace}.
We recall the setting considered here (see the introduction for more on this).

With $0<Q<n$, let $E\subset \R^n$ be an Ahlfors $Q$-regular compact set
with $\diam(E)\le 1$. Let $B$ be a ball in $\R^n$ such that $E\subset\tfrac12 B$. We also assume that 
for each $\alpha\le 0$ there is a measure $\mu_\alpha$ on $B$ such that whenever $\alpha+n-Q>0$
and $x\in B$, and $0<r<2$ such that $r>\dist(x,E)/9$, the comparison $\mu_\alpha(B(x,r))\simeq r^{n+\alpha}$
holds. We also assume that the
ball $B$, equipped with the Euclidean metric $d$ and the measure $\mu_\alpha$,
is doubling and supports a $q$-Poincar\'e inequality for each $q>1$.

\begin{lemma}\label{lem:HLmaximal1}
Suppose that $B$ is a ball in $\R^n$ such that  $E\subset \tfrac12B$, where $E$ is compact and is 
Ahlfors $Q$-regular. Let $\mu_\alpha$ be as in~\eqref{eq:def-mu-alpha} such that $\nu=\mathcal{H}^Q\vert_E$ is
$\alpha+n-Q$-codimensional with respect to $\mu_\alpha$. 
Suppose that $\gamma>0$ such that $\gamma\ge\alpha+n-Q$.
Then there is a constant $C>0$ such that 
whenever $h\in L^1_{loc}(\R^n,\mu_\alpha)$, we have for all $t>0$,
\[
\nu(\{M_\gamma h>t\})\le \frac{C}{t}\int_{B}h\, d\mu_\alpha, 
\]
where $M_\gamma$ is the fractional Hardy-Littlewood maximal function operator given by
\[
M_\gamma h(x)= \sup_{\rad(B)\le 1, x\in B} \rad(B)^\gamma \vint_B h\, d\mu_\alpha.
\]
\end{lemma}

\begin{proof}
This is a variant of the standard proof, the variant being that the measure with respect to which the maximal
operator functions, $\mu_\alpha$, is not the same as the measure $\nu$ with respect to which the superlevel
sets are measured. For this reason we give the complete proof here.

Let $t>0$ and set $E_t:=\{x\in E\, :\, M_\gamma h(x)>t\}$. Then for each $x\in E_t$ there is a ball 
$B_x$ of radius $r_x>0$ such that $x\in B_x$ and 
\[
 r_x^\gamma  \vint_{B_x} h\, d\mu_\alpha>t.
\]
It follows that 
\[
\nu(E\cap B_x)\simeq r_x^Q<\frac{r_x^{\gamma+Q-(\alpha+n)}}{t}\int_{B_x}h\, d\mu_\alpha.
\]
Recalling that $\alpha+n>Q$, we set $\eta=\gamma-(\alpha+n-Q)$. Then $\eta\ge 0$.
The balls $B_x$, $x\in E_t$, cover $E_t$. Therefore, by the $5$-covering theorem (which is
applicable here to $E$ as $\nu$ is doubling on $E$), we obtain a pairwise disjoint countable
subfamily of balls $B_i$with radii $r_i$ such that $E_t\subset \bigcup_i 5B_i$. Then
\begin{align*}
\nu(E_t)\le \sum_i \nu(5B_i)\le C \sum_i r_i^Q \le \frac{C}{t} \sum_i r_i^\eta \int_{B_i}h\, d\mu_\alpha
  &\le \frac{C}{t} \sum_i \int_{B_i}h\, d\mu_\alpha\\
 &\le \frac{C}{t} \int_B h\, d\mu_\alpha,
\end{align*}
where we have used the facts that $r_i\le 1$, $\eta\ge 0$, and that the balls $B_i$ are pairwise disjoint.
\end{proof}

\begin{lemma}\label{lem:HLmaximal2}
Suppose that $0\le g\in L^p(B,\mu_\alpha)$ where $B$ is the ball as in Lemma~\ref{lem:HLmaximal1} and
$1<p<\infty$. Fix $1\le q<p$. Then 
\[
\int_E (M_\gamma(g^q))^{p/q}\, d\nu\le C\int_B g^p\, d\mu_\alpha.
\]
\end{lemma}

\begin{proof}
Recall from the Cavalieri principle that 
\[
\int_E (M_\gamma(g^q))^{p/q}\, d\nu=\frac{p}{q} \int_0^\infty t^{\tfrac{p}{q}-1}\nu(\{Mg^q>t\})\, dt.
\]
For $t>0$, we can write $g^q=G_1+G_2$, where 
\[
G_1=g^q\chi_{\{g^q\le t/2\}},\qquad 
G_2=g^q \chi_{\{g^q>t/2\}}.
\]
Then 
\[
M_\gamma g^q\le M_\gamma G_1+M_\gamma G_2\le \frac{t}{2}+M_\gamma G_2.
\]
As $M_\gamma g^q(z)>t$ when $z\in E_t$, it follows that $\{M_\gamma g^q>t\}\subset \{M_\gamma G_2>t/2\}$. 
Hence by Lemma~\ref{lem:HLmaximal1},
\begin{align*}
\int_E(M_\gamma g^q)^{p/q}\, d\nu &\le\frac{p}{q}\int_0^\infty t^{\tfrac{p}{q}-1}\, \nu(\{M_\gamma G_2>t/2\})\, dt\\
  &\le C\int_0^\infty t^{\tfrac{p}{q}-2} \int_B G_2\, d\mu_\alpha \, dt\\
 & = C \int_0^\infty t^{\tfrac{p}{q}-2} \int_{B\cap\{g^q>t/2\}}g^q\, d\mu_\alpha\, dt\\
&= \int_0^\infty t^{\tfrac{p}{q}-2} \left[ \frac{t}{2} \mu_\alpha(B\cap\{g^q>t/2\})
     +\int_{t/2}^\infty \mu_\alpha(B\cap\{g^q>s\})\, ds\right]\, dt\\
&=C_1\int_0^\infty (t/2)^{\tfrac{p}{q}-1}\mu_\alpha(B\cap \{g^q>t/2\})\, dt\\
  &\qquad\qquad  +C\int_0^\infty\int_0^\infty t^{\tfrac{p}{q}-2}\chi_{(t/2,\infty)}(s)\mu_\alpha(B\cap\{g^q>s\})\, ds\, dt\\
&=C_2\int_B g^p\, d\mu_\alpha\\
&\qquad\qquad
+C\int_0^\infty\left(\int_0^\infty t^{\tfrac{p}{q}-2}\chi_{(0,2s)}(t)\, dt\right)\mu_\alpha(B\cap\{g^q>s\})\, ds\\
&=C_2\int_Bg^p\, d\mu_\alpha+C_3\int_0^\infty s^{\tfrac{p}{q}-1} \mu_\alpha(B\cap\{g^q>s\})\, ds,
\end{align*}
where we also used the Cavalieri principle and Tonelli's theorem in obtaining the last few lines above. By
the Cavalieri principle again, we obtain the desired result.
\end{proof}

Now we are ready to prove Theorem~\ref{thm:main-trace}. For the convenience of the reader, we state an expanded
version of this theorem now.

\begin{thm}\label{thm:Trace}
Let $E$ be an Ahlfors $Q$-regular compact subset of $\tfrac12B$ where $B$ is a ball in $\R^2$. Let $p>1$ and 
$0<\theta<1$ be such that $p\theta<1$. Let $\alpha\le 0$ be such that $\alpha+n-Q>0$ and 
$\theta<1-\tfrac{\alpha+n-Q}{p}$.
Then there exists $C\ge 1$ and a linear trace operator 
\[
T:N^{1,p}(B,\mu_\alpha)\to B^\theta_{p,p}(E,\nu)
\] 
with
\[
\Vert Tu\Vert_{B^\theta_{p,p}(E,\nu)}\le C\, \Vert |\nabla u| \Vert_{L^p(B,\mu_\alpha)}
\] 
and
\[
\Vert Tu\Vert_{L^p(E,\nu)}^p\le C \Vert u\Vert_{N^{1,p}(B,\mu_\alpha)}.
\]
Moreover, if $u\in N^{1,p}(B,\ma)$ is Lipschitz continuous in a neighborhood of $E$, then
$Tu=u\vert_E$.
\end{thm}

Note that if $p\theta<1$, then we can always choose $\alpha\le 0$ satisfying the hypotheses of the above theorem.
Moreover, if we only know that there is a fixed $\alpha>Q-n$ such that $\mu_\alpha$ is doubling and supports
a $p$-Poincar\'e inequality for some $p>1$, then the conclusion of the above theorem holds true as long
as there exists $1\le q<p$ such that $\mu_\alpha$ supports a $q$-Poincar\'e inequality and
$\alpha+n-Q<q(1-\theta)$. The support of a $q$-Poincar\'e inequality for some $1\le q<p$ is guaranteed
by the self-improvement property of Poincar\'e inequality, see~\cite{KZ, HKSTbook}.

\begin{proof}
We first prove the above claim for Lipschitz functions in $N^{1,p}(B,\mu_\alpha)$. 
As $\mu_\alpha$ is doubling and supports a $p$-Poincar\'e inequality, we know that Lipschitz functions
are dense in $N^{1,p}(B,\mu_\alpha)$, and hence we get the estimates for all functions in 
$N^{1,p}(B,\mu_\alpha)$. So, in the following we will assume that $u$ is Lipschitz continuous. It follows
that every point in $B$ is a $\mu_\alpha$-Lebesgue point of $u$. We denote $g_u:=|\nabla u|$.

Let $x,y\in E$ and set $B_0=B(x,2d(x,y))$, and for
positive integers $k$ we set $B_k=B(x,2^{1-k} d(x,y))$ and $B_{-k}=B(y,2^{1-k}d(x,y))$. 
We also set $r_k=\rad(B_k)$ for $k\in\Z$. Then
by the following standard telescoping argument and by the $q$-Poincar\'e inequality, we obtain
\begin{align*}
|u(y)-u(x)| &\le \sum_{k\in\Z}|u_{B_k}-u_{B_{k+1}}|\\
  &\le C \sum_{k\in\Z} \vint_{2B_k}|u-u_{2B_k}|\, d\mu_\alpha\\
&\le C \sum_{k\in\Z} r_k\left(\vint_{2B_k} g_u^q\, d\mu_\alpha\right)^{1/q}\\
&= C \sum_{k\in\Z} r_k^{1-\gamma/q}\left( r_k^\gamma\vint_{2B_k} g_u^q\, d\mu_\alpha\right)^{1/q}\\
&= C\, d(x,y)^{1-\gamma/q} 
   \sum_{k\in\Z} 2^{-|k|(1-\gamma/q)}\left( r_k^\gamma\vint_{2B_k} g_u^q\, d\mu_\alpha\right)^{1/q}.
\end{align*}
By assumption, we have $p\theta<1$. Hence we can choose $\alpha$ with $Q-2<\alpha\le 0$ such that
$p\theta<p-(\alpha+n-Q)$.
Since $\alpha+n-Q>0$, we can choose $\gamma=\alpha+n-Q$ in the above. Then the condition on
$\alpha$ as described above reads as $p\theta<p-\gamma$, and so $0<\theta<1-\gamma/p$. Hence we can
choose $1<q<p$ such that $\theta<1-\gamma/q$, whence by this choice of $q$ we have that 
$1-\gamma/q>0$. It follows that $\sum_{k\in\Z}2^{-|k|(1-\gamma/q)}<\infty$, and so
\[
|u(y)-u(x)|\le C\, d(x,y)^{1-\gamma/q}\left[M_\gamma g_u^q(x)^{1/q}+M_\gamma g_u^q(y)^{1/q}\right].
\]
Hence, for $0<\theta<1$, we obtain
\begin{align*}
\frac{|u(y)-u(x)|^p}{d(x,y)^{\theta p}\nu(B(x,d(x,y)))}
 &\simeq \frac{|u(y)-u(x)|^p}{d(x,y)^{Q+\theta p}}\\
&\le C\, d(x,y)^{p-\tfrac{\gamma}{q}p-\theta p-Q}\, \left[M_\gamma g_u^q(x)^{p/q}+M_\gamma g_u^q(y)^{p/q}\right].
\end{align*}
Therefore,
\begin{align*}
\Vert u\Vert_{B^\theta_{p,p}(E,\nu)}^p
 &\le C \int_E\int_E d(x,y)^{p-\tfrac{\gamma}{q}p-\theta p-Q}\, \left[M_\gamma g_u^q(x)^{p/q}+M_\gamma g_u^q(y)^{p/q}\right]\, d\nu(x)\, d\nu(y)\\
&= C\, [I_1+I_2],
\end{align*}
where
\begin{align*}
I_1&= \int_E\int_E d(x,y)^{p-\tfrac{\gamma}{q}p-\theta p-Q}\, M_\gamma g_u^q(x)^{p/q}\, d\nu(x)\, d\nu(y)\\
I_2&= \int_E\int_E d(x,y)^{p-\tfrac{\gamma}{q}p-\theta p-Q}\, M_\gamma g_u^q(y)^{p/q}\, d\nu(x)\, d\nu(y).
\end{align*}
Thanks to Tonelli's theorem, any estimate we obtain for $I_1$ is valid also for $I_2$, so we consider $I_1$ only next.
Note that for $x\in E$,
\begin{align*}
\int_E d(x,y)^{p-\tfrac{\gamma}{q}p-\theta p-Q}\, d\nu(y)
&=\sum_{j=0}^\infty \int_{B(x,2^{-j})\setminus B(x,2^{-j-1})}d(x,y)^{p-\tfrac{\gamma}{q}p-\theta p-Q}\, d\nu(y)\\
  &\le C \sum_{j=0}^\infty 2^{-j[p-\tfrac{\gamma}{q}p-\theta p-Q]}2^{-jQ}
 = C \sum_{j=0}^\infty 2^{-jp[1-\tfrac{\gamma}{q}-\theta]}.
\end{align*}
As we had chosen $1<q<p$ such that $\theta<1-\gamma/q$, it follows that 
the above series is finite.  
Then we get  from Lemma~\ref{lem:HLmaximal2} that
\[
I_1\le C\int_EM_\gamma g_u^q(x)^{p/q}\, d\nu(x)\le C\int_Bg_u^p\, d\mu_\alpha.
\]
It then follows that 
\[
\Vert u\Vert_{B^\theta_{p,p}(E,\nu)}\le C\, \Vert g_u\Vert_{L^p(B,\mu_\alpha)}=C\, \Vert |\nabla u|\Vert_{L^p(B,\mu_\alpha)}.
\]

In the above computation, if we fix $x\in E$ to be a $\mu_\alpha$--Lebesgue point of $u$
and consider only the balls $B_k=B(x,2^{-k})$ for $k\ge 0$, then we obtain
\begin{align*}
|u(x)|\le |u(x)-u_{B_0}|+|u_{B_0}|
 &\le \sum_{k=0}^\infty |u_{B_k}-u_{B_{k+1}}|+C\, \vint_B|u|\, d\mu_\alpha\\
&\le \sum_{k=0}^\infty r_k\left(\vint_{2B_k}g_u^q\, d\mu_\alpha\right)^{1/q}
     +C\left(\vint_B|u|^p\, d\mu_\alpha\right)^{1/p}\\
&\le C\sum_{k=0}^\infty r_k^{1-\gamma/q}M_\gamma g_u^q(x)^{1/q}
    +C\left(\vint_B|u|^p\, d\mu_\alpha\right)^{1/p}\\
&= C M_\gamma g_u^q(x)^{1/q}
    +C\left(\vint_B|u|^p\, d\mu_\alpha\right)^{1/p}.
\end{align*}
Therefore
\[
|u(x)|^p\le C M_\gamma g_u^q(x)^{p/q} + C\vint_B|u|^p\, d\mu_\alpha.
\]
Integrating the above over $E$ with respect to the measure $\nu$ and applying Lemma~\ref{lem:HLmaximal2}
we obtain
\[
\Vert u\Vert_{L^p(E,\nu)}^p\le C\int_B g_u^p\, d\mu_\alpha+C_B \int_B|u|^p\, d\mu_\alpha,
\]
from which the claim now follows.
\end{proof}

\vskip .3cm

\section{Extension of Besov functions are in Sobolev spaces, or how subsets influence their surroundings}

In this section we show that we can extend functions from the Besov class for $E$ to Newtonian class for 
$(\R^2,\mu_\alpha)$ by proving Theorem~\ref{thm:main-extend}.

Since $E$ is compact and $E\subset\tfrac12B$, we can construct a Whitney cover $B_{i,j}$, 
$i\in\N$ and $j=1,\cdots, M_i$, of $B\setminus E$. 
Such a cover is described in~\cite[Section~2]{HM} and in~\cite[Proposition~4.1.15]{HKSTbook}, 
where the construction did not need the open set $\Omega$ to
be connected, and so their construction is available in our setting as well.
We can ensure
that with $B_{i,j}=B(x_{i,j},r_{i,j})$ we have $r_{i,j}=\dist(x_{i,j},E)=2^{-i}$ and that
for each $T\ge 1$ there exists $N_T\in\N$ such that for each $i\in\N$,
\begin{equation}\label{eq:bdd-overlap}
\sum_{j=1}^{M_i}\chi_{TB_{i,j}}\le N_T.
\end{equation}
Let $\pip_{i,j}$ be a Lipschitz partition of unity subordinate to the cover $B_{i,j}$, that is, 
each $\pip_{i,j}$ is $2^i C$--Lipschitz continuous, $0\le \pip_{i,j}\le 1$, $\text{supp}(\pip_{i,j})\subset 2B_{i,j}$
and
\[
\sum_{i,j}\pip_{i,j}=\chi_{B\setminus E}.
\]
Moreover, there exist $N_1, N_2\in\N$ such that if $2B_{i,j}$ and $2B_{m,n}$ intersect, then $|i-m|<N_1$
and there are at most $N_2$ balls $B_{m,n}$ satisfying the above when $i,j$ are fixed.  For the convenience of
the reader, we state an expanded version of Theorem~\ref{thm:main-extend} below.

\begin{thm}\label{thm:extend}
With $E$, $B$, $\nu=\mathcal{H}^Q\vert_E$ and $0<Q<2$ as above, let $p>1$ and $0<\theta<1$. We fix $\alpha\le 0$ such that
$\alpha+n-Q>0$ and $\theta\ge 1-\tfrac{\alpha+n-Q}{p}$. Then there is a constant $C\ge 1$ and a
linear extension operator
\[
S:B^\theta_{p,p}(E,\nu)\to N^{1,p}(B,\mu_\alpha)
\]
such that 
\[
\int_B |\nabla Su|^p\, d\mu_\alpha\le C\, \Vert u\Vert_{B^\theta_{p,p}(E,\nu)}^p, \qquad 
\int_B|Su|^p\, d\mu_\alpha\le C\, \int_E|u|^p\, d\nu.
\]
Moreover, if $u$ is $L$-Lipschitz on $E$, then $Su$ is $CL$-Lipschitz on $B$.
\end{thm}

Note that in the trace theorem, Theorem~\ref{thm:Trace}, we were not able to gain control of 
$\int_E|Tu|^p\, d\nu$ solely in terms of $\int_B|u|^p\, d\mu_\alpha$. 
The above extension theorem however does allow us these separate controls.

\begin{proof}
From~\cite[Proposition~13.4]{B2S} we know that Lipschitz functions are dense in $B^\theta_{p,p}(E,\nu)$ 
for each $0<\theta<1$ and $p\ge 1$. We fix our attention on $p>1$ and $0<\theta<1$. We will
first extend Lipschitz functions in $B^\theta_{p,p}(E,\nu)$ to $N^{1,p}(B,\mu_\alpha)$ and use this
to conclude that every function in $B^\theta_{p,p}(E,\nu)$ has an extension lying in $N^{1,p}(B,\ma)$.
To this end, let $u\in B^\theta_{p,p}(E,\nu)$ be Lipschitz continuous, and for $x\in B\setminus E$ we set
\[
Su(x)=\sum_{i,j} u_{2B_{i,j}}\, \pip_{i,j}(x),
\]
where $u_{2B_{i,j}}:=\vint_{2B_{i,j}}u\, d\nu$.
We extend $Su$ to $E$ by setting $Su(x)=u(x)$ when $x\in E$.
If $x\in E$ and $y\in B_{i_0,j_0}$, then 
\begin{align*}
|Su(y)-u(x)|&=\bigg\vert \sum_{i,j}[u_{2B_{i,j}}-u(x)]\pip_{i,j}(y)\bigg\vert\\
  &\le \sum_{i,j}\pip_{i,j}(y)\, \vint_{2B_{i,j}}|u(w)-u(x)|d\nu(w)\\
  &\le \sum_{i,j}\pip_{i,j}(y)\, L\, 2^{1-i}\ 
  \le CL\, d(y,x). 
\end{align*}
It follows that for each $x\in E$, we have 
\[
\limsup_{r\to 0^+}\vint_{B(x,r)\setminus E}|Su(y)-u(x)|^q\, d\mu_\alpha(y)=0
\]
for all $1\le q<\infty$.

If $x,y\in B_{i_0,j_0}$, then by the properties of the Whitney cover listed above,
\begin{align}
|Su(y)-Su(x)|&=\bigg\vert\sum_{i,j}[u_{2B_{i,j}}-u_{2B_{i_0,j_0}}][\pip_{i,j}(y)-\pip_{i,j}(x)]\bigg\vert\notag \\
 \le &\frac{C\, d(y,x)}{r_{i_0}}\sum_{i,j; 2B_{i,j}\cap B_{i_0,j_0}\ne \emptyset}|u_{2B_{i,j}}-u_{2B_{i_0,j_0}}|\notag \\
 \le &\frac{C\, d(y,x)}{r_{i_0}}\sum_{i,j; 2B_{i,j}\cap B_{i_0,j_0}\ne \emptyset}
   \vint_{2B_{i,j}}\vint_{2B_{i_0,j_0}}|u(w)-u(v)|\, d\nu(v)\, d\nu(w)\notag\\
\le &\frac{C\, d(y,x)}{r_{i_0}}\vint_{CB_{i_0,j_0}}\vint_{CB_{i_0,j_0}}|u(w)-u(v)|\, d\nu(v)\, d\nu(w)\label{eq:control1}\\
\le &\frac{C\, d(y,x)}{r_{i_0}} 2CL\, r_{i_0} 
\le C\, L\, d(x,y). \label{eq:control2}
\end{align}
It follows that $Su$ is $CL$-Lipschitz on $B$ and hence is in $N^{1,p}(B,\mu_\alpha)$. It now only remains to obtain
norm bounds.

Recall from~\cite[Theorem~5.2 and equation~(5.1)]{GKS} that 
\begin{equation}\label{eq:Besov-alt-norm}
\Vert u\Vert_{B^\theta_{p,p}(E,\nu)}^p
  \simeq \sum_{n=0}^\infty \int_E \vint_{B(x,2^{-n})} \frac{|u(x)-u(y)|^p}{2^{-n\theta p}}\, d\nu(y)\, d\nu(x).
\end{equation}
As $E$ is Ahlfors $Q$ regular for some $Q<2$, it follows that $\mathcal{H}^2(E)=0$, and hence $\mu_\alpha(E)=0$. 
Let $z\in B_{i_0,j_0}$. Setting $x=z$ and letting $y\to z$, by applying the H\"older inequality to~\eqref{eq:control1} 
we have that
\begin{align*}
\Lip Su(z)^p&=\left(\limsup_{y\to z}\frac{|Su(y)-Su(z)|}{d(y,z)}\right)^p\\
 &\le \frac{C}{r_{i_0}^p}  \vint_{CB_{i_0,j_0}}\vint_{CB_{i_0,j_0}}|u(w)-u(v)|^p\, d\nu(v)\, d\nu(w)\\
 &= \frac{C}{2^{-i_0(1-\theta)p}}\vint_{CB_{i_0,j_0}}\vint_{CB_{i_0,j_0}}\frac{|u(w)-u(v)|^p}{2^{-i_0\theta p}}\, d\nu(w)\, d\nu(v)\\
 &\le \frac{C}{2^{-i_0(Q+(1-\theta)p)}}\int_{CB_{i_0,j_0}}\vint_{B(v,2^{k_0-i_0})}\frac{|u(w)-u(v)|^p}{2^{-i_0\theta p}}\, d\nu(w)\, d\nu(v),
\end{align*}
where $k_0$ is the smallest positive integer such that $2^{k_0}\ge 2C$; note that $k_0$ is independent of $i_0,j_0, v$.
Here we also used the fact that $r_{i_0}\simeq \dist(z,E)\simeq 2^{-i_0}$. Integrating the above over $B_{i_0,j_0}$, we obtain
\begin{align*}
\int_{B_{i_0,j_0}}&\Lip Su(z)^p\, d\mu_\alpha(z)=\int_{B_{i_0,j_0}}|\nabla Su(z)|^p\, d\ma(z)\\
  &\le C\, 2^{-i_0(\alpha +2-Q-(1-\theta)p)} \int_{CB_{i_0,j_0}}\vint_{B(v,2^{k_0-i_0})}\frac{|u(w)-u(v)|^p}{2^{-i_0\theta p}}\, d\nu(w)\, d\nu(v).
\end{align*}
Summing the above over $j_0=1,\cdots, M_{i_0}$ and noting by~\eqref{eq:bdd-overlap}
that $\sum_{j=1}^{M_{j_0}}\chi_{CB_{i_0,j}}\le N_C$ with $E\subset\bigcup_{j=1}^{M_{i_0}}CB_{i_0,j}$, 
and then summing over $i_0$, we obtain
\begin{align*}
\int_B&\Lip Su(z)^p\, d\mu_\alpha(z)\\
&\le C \sum_{i_0=0}^\infty 2^{-i_0(\alpha +2-Q-(1-\theta)p)} 
  \int_E \vint_{B(v,2^{k_0-i_0})}\frac{|u(w)-u(v)|^p}{2^{-i_0\theta p}}\, d\nu(w)\, d\nu(v)\\
 &\le \sum_{i_0=0}^\infty \int_E \vint_{B(v,2^{k_0-i_0})}\frac{|u(w)-u(v)|^p}{2^{-i_0\theta p}}\, d\nu(w)\, d\nu(v)
\end{align*}
provided that $\alpha+n-Q-(1-\theta)p\ge 0$. So if $\alpha\le 0$ is chosen such that $\alpha+n-Q>0$ and 
\[
\theta\ge 1-\frac{\alpha+n-Q}{p},
\]
then by~\eqref{eq:Besov-alt-norm} we have that
\[
\int_B|\nabla Su|^p\, d\ma=\int_B \Lip Su^p\, d\mu_\alpha\le C\, \Vert u\Vert_{B^\theta_{p,p}(E,\nu)}^p.
\] 

To complete the argument, we next obtain control of $\Vert Su\Vert_{L^p(B,\mu_\alpha)}$. For $x\in B_{i_0,j_0}$,
\begin{align*}
|Su(x)|&=\bigg\vert\sum_{i,j: 2B_{i,j}\cap B_{i_0,j_0}\ne \emptyset}\pip_{i,j}(x)\vint_{2B_{i,j}}u(y)\, d\nu(y)\bigg\vert\\
&\le C \vint_{C_0B_{i_0,j_0}}|u(y)|\, d\nu(y)\\
&\le C \left(\vint_{C_0B_{i_0,j_0}}|u(y)|^p\, d\nu(y)\right)^{1/p}.
\end{align*}
Therefore
\begin{align*}
\int_{B_{i_0,j_0}}|Su(x)|^p\, d\mu_\alpha(x)
  &\le C\, \mu_\alpha(B_{i_0,j_0}) \vint_{C_0B_{i_0,j_0}}|u(y)|^p\, d\nu(y)\\
  &\le C\, 2^{-i_0(\alpha+n-Q)}\int_{C_0B_{i_0,j_0}}|u(y)|^p\, d\nu(y).
\end{align*}
As before, summing over $j_0=1,\cdots, M_{i_0}$ and then over $i_0$ gives
\[
\int_B|Su(x)|^p\, d\mu_\alpha(x)\le C \sum_{i_0=0}^\infty 2^{-i_0(\alpha+n-Q)} \int_E |u(y)|^p\, d\nu(y).
\]
As $\alpha+n-Q>0$, it follows that
\[
 \int_B|Su(x)|^p\, d\mu_\alpha(x)\le C \int_E |u(y)|^p\, d\nu(y)
\]
as desired.
\end{proof}

%
%
%

\vskip .5cm

\noindent Address:\\

\vskip .2cm

Department of Mathematical Sciences, University of Cincinnati, P.O. Box~210025, Cincinnati, OH 45221-0025, USA.\\

\noindent E-mail:  J.L.:~{\tt jlindquistmath@gmail.com}, \ N.S.:~{\tt shanmun@uc.edu} 

\end{document}